\documentclass[12pt]{amsart}
\usepackage{amsmath,amssymb,latexsym}
\usepackage{fullpage}

\theoremstyle{plain}
\newtheorem{theorem}{Theorem}
\newtheorem{proposition}{Proposition}
\newtheorem{lemma}{Lemma}

\newtheorem{corollary}{Corollary}

\begin{document}

\title{Generating Functions on Covering Groups}
\author{David Ginzburg}
\address{Ginzburg: School of Mathematical Sciences, Tel Aviv University, Ramat Aviv, Tel Aviv 6997801,
Israel}
\email{ginzburg@post.tau.ac.il}

\thanks{The author is partly supported by the Israel Science
Foundation grant number  259/14.}

\begin{abstract}
In this paper we prove Conjecture 1.2 in \cite{B-F}. This enables us
to establish the meromorphic continuation of the standard partial $L$
function $L^S(s,\pi^{(n)})$. Here, $\pi^{(n)}$ is a genuine irreducible cuspidal representation of the group $GL_r^{(n)}({\bf A})$. 

\end{abstract}

\maketitle

\section{introduction}
Let ${\bf A}$ denote the ring of adeles of a global field $F$. Assume 
that $F$ contains the $n$ roots of unity. Let $GL_r^{(n)}({\bf A})$
denote the $n$ fold metaplectic cover of the group $GL_r({\bf A})$.
Let $\pi^{(n)}$ denote a genuine irreducible cuspidal representation of $GL_r^{(n)}({\bf A})$. As explained in \cite{B-F}, to this 
representation one can attach the partial standard $L$ function, 
denoted by $L^S(s,\pi^{(n)})$.

In \cite{B-F} the authors introduced a global integral to study this $L$ function. 
Their idea was to take the well known Rankin-Seleberg convolution of two
cuspidal representation of $GL_r({\bf A})$ and $GL_n({\bf A})$.  For the constructions of these Rankin-Selberg integrals, and the motivations for the adopted construction in the covering groups, we refer to \cite{B-F}. Then, they 
replaced the cuspidal representation of $GL_n({\bf A})$ by the Theta
representation $\Theta_n^{(n)}$. This last representation was constructed in \cite{K-P}. For example, if $r>n$, the construction suggested in \cite{B-F} is given by 
\begin{equation}\label{global100}
\int\limits_{GL_r(F)\backslash GL_r({\bf A})}
\int\limits_{V_{r,n}(F)\backslash V_{r,n}({\bf A})}
\phi\left (v\begin{pmatrix}
g&\\ &I_{r-n}\end{pmatrix}\right )\overline{\theta(g)}\psi_{V_{r,n}}(v)|\text{det} g|^{s-\frac{r-n}{2}}dvdg
\end{equation}
Here $\phi$ is a vector in the space of $\pi^{(n)}$, and $\theta$
is a vector in the space of $\Theta_n^{(n)}$. The group $V_{r,n}$, and
the character $\psi_{V_{r,n}}$ are defined in the last section.
In fact in \cite{B-F} the authors concentrated on the case when $r<n$,
but up to some modifications as explained in \cite{B-F} Section 2, 
the idea is the same. 

A straight-forward unfolding, implies that for $\text{Re}(s)$ large, integral \eqref{global100} is equal to 
\begin{equation}\label{global110}
\int\limits_{V_r({\bf A})\backslash GL_r({\bf A})}W_\phi\begin{pmatrix} g&\\ &I_{r-n}\end{pmatrix}\overline{W_\theta(g)}|\text{det} g|^{s-\frac{r-n}{2}}dg
\end{equation}
Here $W_\phi$ denotes the Whittaker coefficient of $\phi$, and similarly we define $W_\theta$. 

Since the Whittaker coefficient $W_\phi$ is not factorizable, it is
not obvious that the above integral represents an Euler product. To
show that it does, one needs to apply the method referred to as
the New Way which was developed in \cite{PS-R}. See \cite{B-F} for a discussion and references for this method. 

In our context, to deduce that integral \eqref{global110} is Eulerian
one can proceed in two steps. The first step is to find a generating 
function for the unramified local $L$ function $L(s,\pi^{(n)})$. See
\cite{B-F} page 5 for its definition. In \cite{B-F}, such a function
was introduced, and was  denoted by $\widetilde{\Delta}_s$. The second step  is to use this function to compute a local integral which is obtained from integral \eqref{global110}, and to show that this local integral is
independent of the choice of the local Whittaker function attached 
to the local component of $\phi$. To do that one needs to
compute the Whittaker function of $\widetilde{\Delta}_s$. This is
done in \cite{B-F} for the cases $r=n=2,3$. The general case is conjectured in \cite{B-F} Conjecture 1.2. 

In this paper we prove Conjecture 1.2 of \cite{B-F} in complete 
generality. To do this we give a different realization for the function
$\widetilde{\Delta}_s$. This realization makes the proof of the
stated conjecture relatively simple. The new realization is described in Section \ref{gen} and is given by a certain unique functional 
defined on the local Theta representation $\Theta_{nr}^{(n)}$ defined
on the group $GL_{nr}^{(n)}$. We then use this functional to define
a function on the group $GL_{nr}^{(n)}$, which we denote by $W_{nr}^{(n)}(h)$. Here $h\in GL_{nr}^{(n)}$. Restricting to the group $GL_r^{(n)}$ we obtain a function
on that group which we use to give the new expression for $\widetilde{\Delta}_s$. Thus our result contains two parts. The first is the proof that the function $W_{nr}^{(n)}(h)$ restricted to $GL_r
^{(n)}$ is indeed the generating function for the standard $L$ function. This we do in Proposition \ref{prop2}. The second, and
the main result of this paper, is to obtain the desired expression
for the Whittaker function of the generating function. This we do 
in Theorem \ref{th1}, which is Conjecture 1.2 in \cite{B-F}. In both cases the computations are quite
straight forward and are done by a repeated application of Lemma
\ref{simple1} and Corollary \ref{cor1} stated and proved in Subsection 
\ref{loc1}. 

As mentioned above, the global result and some of the computations done in \cite{B-F} assumes that $r<n$. This is just a technical point.
The authors of \cite{B-F} were well aware that
their construction works for all $r$ and $n$. To complete their 
result, in the last section we give some details in the other two cases, that is the case when $r>n$ and $r=n$. 

To summarize, combining \cite{B-F} with our result we have
\begin{theorem}\label{main10}
Let $\pi^{(n)}$ denote an irreducible cuspidal representation of the group $GL_r^{(n)}({\bf A})$. Then the partial $L$ function 
$L^S(s,\pi^{(n)})$ has a meromorphic continuation to the whole complex
plane. When $r\ne n$ this partial $L$ function is holomorphic. When
$r=n$ it can have at most a simple pole at $s=1$.
\end{theorem}

As a final remark we mention that in fact we do expect that the partial $L$ function $L^S(s,\pi^{(n)})$ will also be holomorphic in
the case when $r=n$. From the global integral given in the last
section we deduce that if this $L$ function has a simple pole at $s=1$, then $\pi^{(n)}$ will be isomorphic to $\Theta_n^{(n)}$. This
we believe cannot happen.

\section{ Generating functions}\label{gen}
The main references for this section are \cite{B-F} and \cite{K-P}. 
Fix a positive integer $n>1$. Let
$F$ denote a local nonarchimedean field which contains the $n$ roots
of unity. Let $GL_r^{(n)}$ denote the metaplectic $n$ fold cover 
of the group $GL_r$. Let $\pi^{(n)}$ denote an unramified 
representation of $GL_r^{(n)}$ attached to a character $\chi$
of the torus $T_r$ of $GL_r$. This is explained in \cite{B-F} page 5.
Assuming that $\chi$ is in general position, one can attach to
$\pi^{(n)}$, the local $L$ function which is defined as 
\begin{equation}\label{local1}
L(s,\pi^{(n)})=\frac{1}{\prod_{i=1}^r(1-\chi_i^n(p)q^{-s})}
\end{equation}
Here $p$ is a generator of the maximal ideal in the ring of integers
of $F$, and $q^{-1}=|p|_F$. Also, $s$ is a complex variable. 

In \cite{B-F} formula (1.4) the function $\widetilde{\Delta}_s(h)$
is defined. This is a function of $GL_r^{(n)}$, and Proposition 1.1
in \cite{B-F} states that for $\text{Re}(s)$ large, 
\begin{equation}\label{local2}
\int\limits_{GL_r}\omega_{\pi^{(n)}}(h)\widetilde{\Delta}_s(h)dh=
L(s,\pi^{(n)})
\end{equation}
Here, $\omega_{\pi^{(n)}}$, denoted by $\sigma$ in \cite{B-F}, is the spherical function attached to $\pi^{(n)}$. Thus $\omega_{\pi^{(n)}}$
is a $K_r$  bi-invariant function of $GL_r^{(n)}$ where $K_r$ is the
maximal compact subgroup of $GL_r$ embedded in $GL_r^{(n)}$ as described in \cite{B-F}. As is well known, the function
$\widetilde{\Delta}_s(h)$ is uniquely determined by Proposition 1.1
in \cite{B-F}. This function is referred to as the generating function 
for the standard $L$ function of the group $GL_r^{(n)}$.

We will give a different realization of the function 
$\widetilde{\Delta}_s(h)$. To do that let $\Theta_{nr}^{(n)}$ denote
the local unramified Theta representation of $GL_{nr}^{(n)}$ as constructed in  \cite{K-P}. Thus, this representation is the unramified sub-representation of $Ind_{B_{nr}^{(n)}}^{GL_{nr}^{(n)}}\delta_{B_{nr}}^{\frac{n-1}{2n}}$. Here, $B_{nr}$ is the Borel subgroup of $GL_{nr}$. This representation is not generic, however
it still has a certain unique functional defined on it. To describe
this functional, let $U_{nr}$ denote the unipotent radical of the
parabolic subgroup of $GL_{nr}$ whose Levi part is $GL_r\times GL_r
\times\ldots\times GL_r$.
In term of matrices the group $U_{nr}$ consists of all matrices of the
form
\begin{equation}\label{mat1}
\begin{pmatrix} I&X_{1,2}&X_{1,3}&\dots&X_{1,n}\\
&I&X_{2,3}&\dots&X_{2,n}\\
&&I&\ddots&\vdots\\ &&&\ddots&X_{n-1,n}\\ &&&&I\end{pmatrix}
\end{equation}
Here $I$ is the $r\times r$ identity  matrix, and $X_{i,j}\in 
Mat_{r\times r}$.

Let $\psi$ denote an unramified character of $F$. Define a character
$\psi_{U_{nr}}$ of $U_{nr}$ as follows. For $u\in U_{nr}$ as above, define $\psi_{U_{nr}}(u)=\psi(\text{tr}(X_{1,2}+X_{2,3}+\cdots +
X_{n-1,n}))$. The stabilizer of $\psi_{U_{nr}}$ inside $GL_r\times GL_r \times\ldots\times GL_r$ is the group $GL_r^\Delta$ embedded diagonally. The embedding of $GL_r^\Delta$ inside $GL_{nr}$ is given
by $g\mapsto \text{diag}(g,g,\ldots,g)$. 

Given a representation $\sigma_{nr}^{(n)}$ of $GL_{nr}^{(n)}$, we
consider the space of all functionals  on $\sigma_{nr}^{(n)}$ which
satisfies $l(\sigma_{nr}^{(n)}(u)v)=\psi_{U_{nr}}(u)l(v)$ for all
$u\in U_{nr}$ and all vectors $v$ in the space of $\sigma_{nr}^{(n)}$.
Given such a functional, we may consider the space of functions
$W_v^{(n)}(h)=l(\sigma_{nr}^{(n)}(h)v)$. 

Henceforth we shall assume that $\sigma_{nr}^{(n)}=\Theta_{nr}^{(n)}$ 
and denote the corresponding space of functions by $W_{nr}^{(n)}(h)$.
Then, the following proposition is proved in \cite{C} Theorem 1.2,
\begin{proposition}\label{prop1}
The space of functionals $l$ defined as above on the representation
$\Theta_{nr}^{(n)}$ is one dimensional. Moreover, for all $k\in K_r$
viewed as a subgroup of $GL_r^\Delta$, we have  $W_{nr}^{(n)}(kh)=W_{nr}^{(n)}(h)$. 
\end{proposition}
It is not hard to construct the space of functions $W_{nr}^{(n)}(h)$ explicitly on the space of $\Theta_{nr}^{(n)}$. Indeed, let $f\in Ind_{B_{nr}^{(n)}}^{GL_{nr}^{(n)}}\delta_{B_{nr}}^{\frac{n-1}{2n}}$. Let $U_{nr}^0$ denote the subgroup of $U_{nr}$ which consists of all matrices
$u$ as in \eqref{mat1} such that $X_{i,j}\in Mat_{r\times r}^0$ for all $i$ and $j$.
Here $Mat_{r\times r}^0$ is the subgroup of $Mat_{r\times r}$ 
consisting of all matrices $X$ such that $X[l_1,l_2]=0$ for all $l_1<l_2$, where $X[l_1,l_2]$ denotes the $(l_1,l_2)$-th entry of $X$. Then
\begin{equation}\label{local3}
W_{nr}^{(n)}(h)=\int\limits_{U_{nr}^0}f(w_Jw_0uh)\psi_{U_{nr}}(u)du
\end{equation}
defines the space of functions which satisfies the required transformation  properties, provided it is not identically zero. Here $w_J$ is
the Weyl element $w_J=\text{diag}(J_n,J_n,\ldots,J_n)\in GL_{nr}$
where $J_n$ is the longest Weyl element of $GL_n$. The Weyl element $w_0$ is defined as the
element whose $(a+bn, (a-1)r+b+1)$ entry is one for all $1\le a\le n$
and $0\le b\le r-1$, and zero elsewhere. Matrix multiplication implies that $w_Jw_0$ is the shortest Weyl element of $GL_{nr}$ with the 
property that for all $u\in U_{nr}^0$, we have that $wuw^{-1}$ is a lower unipotent matrix. By  considering the function $W_{nr}^{(n)}(h)$ corresponding to the $K_{nr}$ fixed vector $f$ in the space of $\Theta_{nr}^{(n)}$, one can easily show that
$W_{nr}^{(n)}(e)\ne 0$. This will follow from the computation which we will perform in the next Proposition. 

Before doing that, it will be convenient to perform a simple 
computation  which we will refer to several times. We will do it in
the following subsection,
\subsection{ A local computation}\label{loc1}
Let $F$ denote a local field. Given a root $\alpha$ associated with the group $GL_b$, we will denote by
$x_\alpha(l)$ the one dimensional unipotent subgroup of $GL_b$ 
associated with this root. Assume that $\alpha$ and $\beta$ are two
roots such that $\alpha+\beta$ is also a root. Thus, we have
$x_\alpha(z)x_\beta(l)=x_\beta(l)x_\alpha(z)x_{\alpha+\beta}(lz)$.
Let $h(a)$ denote a one dimensional torus of $GL_b$ which satisfies
the property $h(a)^{-1}x_\alpha(z)h(a)=x_\alpha(a^{-1}z)$ for all $a\in F^*$.

Let $f$ denote a function defined on $GL_b(F)$ which satisfies the
property 
\begin{equation}\label{property1}
f(x_\beta(l_1)x_{\alpha+\beta}(l_2)gk)=\psi(-l_2)f(g)
\end{equation}
for all $k\in K_b$, where $K_b$ is the standard maximal compact 
subgroup of $GL_b$. 

Our goal in this subsection is to compute the integral
\begin{equation}\label{locint1}\notag
I=\int\limits_{F^2}f(x_\alpha(z)x_\beta(l)h(a))\psi(\epsilon l)dzdl
\end{equation}
Here $\epsilon=0,-1$. We have
\begin{lemma}\label{simple1}
We have, $I=f(h(a)x_\alpha(-a^{-1}\epsilon))=f(x_\alpha(-\epsilon)
h(a))$.
\end{lemma}
\begin{proof}
Since $f$ is right $K_b$ invariant, then
$$I=\int\limits_{F^2}\int\limits_{|m|\le 1}f(x_\alpha(z)x_\beta(l)h(a)x_\alpha(m))\psi(\epsilon l)dmdzdl$$
Conjugating $x_\alpha(m)$ to the left, and using the above 
assumptions on the commutation relations and properties \eqref{property1}, we obtain the integral
$\int \psi(-lam)dm$ as inner integration. Here $m$ is integrated over
$|m|\le 1$. Hence we may restrict the integration domain in integral
$I$ to the domain $|la|\le 1$.

The next step is to conjugate $x_\beta(l)$ to the left. Using the
commutation relations and properties \eqref{property1}, we obtain
$$I=\int\limits_{F}f(x_\alpha(z)h(a))\int\limits_{|la|\le 1}\psi(zl+\epsilon l)dldz$$
Changing variables in $l$, we obtain
$$I=|a|^{-1}\int\limits_{|(z+\epsilon)a^{-1}|\le 1}f(x_\alpha(z)h(a))dz=
|a|^{-1}\int\limits_{|(z+\epsilon)a^{-1}|\le 1}f(h(a)x_\alpha(a^{-1}z))dz$$
Writing $a^{-1}z=a^{-1}z+\epsilon a^{-1}-\epsilon a^{-1}$, we obtain
$$I=|a|^{-1}f(h(a)x_\alpha(-a^{-1}\epsilon))\int\limits_{|(z+\epsilon)a^{-1}|\le 1}dz$$
from which the Lemma follows.

\end{proof}

With the above notation we prove,
\begin{corollary}\label{cor1}
We have 
\begin{equation}\label{unra1}\notag
\int\limits_F f(x_\alpha(z))dz=f(e)
\end{equation}
\end{corollary}

\begin{proof}
Since $f$ is right invariant under $K_b$, the above integral is equal to
$$\int\limits_F\int\limits_{|m|\le 1} f(x_\alpha(z)x_\beta(m))dmdz$$
Conjugating $x_\beta(m)$ to the left, we obtain from the left invariant properties of $f$, the integral $\int \psi(mz)dm$ as inner
integration. Here $m$ is integrated over $|m|\le 1$. The result follows.

\end{proof}

\subsection{On the generating function}\label{loc2}
In this subsection we will prove that the generating function can be expressed in term of the function $W_{nr}^{(n)}$.

Embed $g\in GL_r$ in $GL_{nr}$ as $g\mapsto g_0=\text{diag}(g,I_r,
\ldots,I_r)$. We have
\begin{proposition}\label{prop2}
Let $W_{nr}^{(n)}(h)$ denote the function corresponding to the $K_{nr}$ fixed vector. Then, for $s'=\frac{s}{n}-\frac{(n-2)r}{2}-
\frac{1}{2n}$ we have
\begin{equation}\label{local4}
\widetilde{\Delta}_s(g)=\overline{W_{nr}^{(n)}(g_0)}|\text{det}g|^{s'}
\end{equation}
\end{proposition} 
\begin{proof}
It follows from Proposition \ref{prop1} that the function 
$W_{nr}^{(n)}(g_0)$ is $K_r$ bi-invariant. To prove the Proposition it is enough to show that
\begin{equation}\label{local5}
\int\limits_{GL_r}\omega_{\pi^{(n)}}(g)\overline{W_{nr}^{(n)}(g_0)}|\text{det}g|^{s'}dg=L(s,\pi^{(n)})
\end{equation}
Notice that this will also imply that the function $W_{nr}^{(n)}(h)$ as defined in \eqref{local3} is not identically zero on the space of the representation
$\Theta_{nr}^{(n)}$. Using the identity $\omega_{\pi^{(n)}}(g)=\int_{K_r}f_{\pi^{(n)}}(kg)dk$ we may, after a change of variables, replace in \eqref{local5} the function $\omega_{\pi^{(n)}}$ by 
$f_{\pi^{(n)}}$. Here $f_{\pi^{(n)}}$ is the unramified vector in the
space of $\pi^{(n)}$. Performing the Iwasawa decomposition, the integral in equation \eqref{local5} is equal to
\begin{equation}\label{local6}
\int\limits_{T_r}f_{\pi^{(n)}}(t)\int\limits_{V_r}\overline{W_{nr}^{(n)}(v_0t_0)}|\text{det}\ t|^{s'}\delta_{B_r}(t)^{-1}dv_0dt
\end{equation}
Here $V_r$ is the maximal unipotent subgroup of $GL_r$ consisting
of upper unipotent matrices. Also, $t=\text{diag}(a_1,a_2,\ldots,a_r)$. Plug integral \eqref{local3} in
integral \eqref{local6}. Thus,  we obtain the integral 
\begin{equation}\label{local7}
\int\limits_{V_r}\int\limits_{U_{nr}^0}
\overline{f}(w_Jw_0uv_0t_0)\psi_{U_{nr}}(u)dudv_0
\end{equation}
as an inner integration to integral \eqref{local6}. Let $U_{nr}^1$
denote the subgroup of $U_{nr}^0$ consisting of all matrices such that
$X_{1,2}[i,j]=0$ for all $i\ne j$. We claim that integral \eqref{local7} is equal to
\begin{equation}\label{local8}
\int\limits_{U_{nr}^1}
\overline{f}(w_Jw_0ut_0)\psi_{U_{nr}}(u)du
\end{equation}
We do this by using Lemma \ref{simple1} several times. It is
convenient to use the following notations.
For all integers $1\le a,b\le nr$ and all $m\in F$,
let $x_{a,b}(m)=I_{nr}+me_{a,b}$. Here $e_{a,b}$ is the matrix of 
size $nr$ which has a one in the $(a,b)$ entry, and zero elsewhere.

In integral \eqref{local7} consider the integrations over the variables $X_{1,2}[r,r-1]$ and $v_0[r-1,r]$, where the last variable indicates the $(r-1,r)$ entry of $v_0$. In the notations of Subsection
\ref{loc1}, let $x_\alpha(z)=x_{r,r-1}(z)$ where $z=X_{1,2}[r,r-1]$,
and let $x_\beta(l)=x_{r-1,r}(l)$ where $l=v_0[r-1,r]$. With these
notations we have $x_{\alpha+\beta}(m)=x_{r-1,2r-1}(m)$ and from the
definition of the character $\psi_{U_{nr}}$ we have $\psi_{U_{nr}}(
x_{\alpha+\beta}(m))\ne 1$. Hence, all the conditions of Lemma \ref{simple1} are satisfied with $\epsilon=0$ and $h(a)=h(a_r)=
\text{diag}(I_{r-1},a_r,I_{nr-r})$. From this we deduce that in the integral 
\eqref{local7}, we may restrict the domain of integration to the
group $V_r$ with the condition that $v_0[r-1,r]=0$, and to the group $U_{nr}^0$ with the condition $X_{1,2}[r,r-1]=0$.

In general, we apply this process in the following order. Fix $r+1\le j\le 2r-1$. Then for all $j-r+1\le i\le r$, set $x_\alpha(z)=x_{i,j}(z)$ with $z=X_{1,2}[i,j-r]$, and $x_\beta(l)=x_{j-r,i}(l)$ with 
$l=v_0[j-r,i]$. With these notations we have $x_{\alpha+\beta}(m)=
x_{j-r,j}(m)$. Since $\psi_{U_{nr}}$ is not trivial on $x_{\alpha+\beta}(m)$, we can apply Lemma \ref{simple1} with $\epsilon=0$. The 
end result of this repeated process is that integral \eqref{local7}
is equal to integral \eqref{local8}.

Conjugating by $w_0$, write $w_0U_{nr}^1w_0^{-1}=U_{nr}^2U_{nr}^3$, where the  groups $U_{nr}^2$ and $U_{nr}^3$ are defined as follows.
First, identify the group  $U_{nr}^2$ with $r$ copies of the group $V_n$. Here $V_n$ is defined to be the group of all upper unipotent matrices of $GL_n$. The embedding of $U_{nr}^2$ inside $GL_{nr}$ is
given by $(v_{n,1}, v_{n,2}, \ldots, v_{n,r})\mapsto \text{diag}(v_{n,1}, v_{n,2},\ldots, v_{n,r})$. Here $v_{n,i}\in V_n$. 
To define the group $U_{nr}^3$, consider the unipotent group generated by all matrices of the form
\begin{equation}\label{mat2}
\begin{pmatrix} I&&&&\\ Y_{2,1}&I&&&\\ Y_{3,1}&Y_{3,2}&I&&\\
\vdots&\vdots&\ddots&I&\\ Y_{r,1}&Y_{r,2}&\dots&Y_{r,r-1}&I
\end{pmatrix}
\end{equation}
Here $Y_{i,j}$ is in $Mat_n$. Then the group $U_{nr}^3$ is generated
by all matrices as in \eqref{mat2} which satisfies the conditions
$Y_{i,j}[l_1,l_2]=Y_{i,j}[1,2]=0$ for all $l_1\ge l_2$.

For $v\in V_n$, let $\psi_{V_n}(v)$ denote the Whittaker 
character of the group $V_n$. This character is defined as follows.
Given $v=(v[i,j])\in V_n$, then 
\begin{equation}\label{whit1}
\psi_{V_n}(v)=\psi(v[1,2]+v[2,3]+\cdots +v[n-1,n])
\end{equation}
Let $u_2=\text{diag}(v_{n,1}, v_{n,2},\ldots, v_{n,r})\in U_{nr}^2$.
Define the character $\psi_{U_{nr}^2}$ of $U_{nr}^2$ as 
$\psi_{U_{nr}^2}(u_2)=\psi_{V_n}(v_{n,1})\psi_{V_n}(v_{n,2})\ldots
\psi_{V_n}(v_{n,r})$. Then, in the notations of  $U_{nr}^2U_{nr}^3$
the character $\psi_{U_{nr}}$ transforms to the character 
$\psi_{U_{nr}^2}$ on the group $U_{nr}^2$, and is trivial on the group 
$U_{nr}^3$.

Thus, integral \eqref{local8} is equal to
\begin{equation}\label{local11}
\int\limits_{U_{nr}^3}
f_W(u_3w_0t_0w_0^{-1})du_3
\end{equation}
where
\begin{equation}\label{local12}
f_W(h)=\int\limits_{U_{nr}^2}\overline{f}(w_Ju_2h)\psi_{U_{nr}^2}(u_2)
du_2\notag
\end{equation}
We have $w_0t_0w_0^{-1}=\text{diag} (A_1,A_2,\ldots,A_r)$ where
$A_i=\text{diag}(a_i,I_{n-1})$. Conjugating the matrix $w_0t_0w_0^{-1}$ to the left in integral \eqref{local11} we obtain the factor
\begin{equation}\label{fact1}
\alpha(t)=(|a_2||a_3|^2|a_4|^3\ldots |a_r|^{r-1})^{n-2}\notag
\end{equation}
from the change of variables in $U_{nr}^3$. Thus, integral \eqref{local11} is equal to
\begin{equation}\label{local13}
\alpha(t)\int\limits_{U_{nr}^3}f_W(w_0t_0w_0^{-1}u_3)du_3
\end{equation}
We claim that integral \eqref{local13} is equal to $\alpha(t)
f_W(w_0t_0w_0^{-1})$. This we will show by a repeated application of 
Corollary \ref{cor1}. Indeed, fix $2\le i\le r$, where we first start
with $i=r$, then $i=r-1$ and so on. Let $1\le k\le i-1$. Assume
that $l_1$ and $l_2$ are such that $Y_{i,k}[l_1,l_2]$ is a variable in the domain of integration in integral \eqref{local13}. In the
notations of Subsection \ref{loc1}, let $x_\alpha(z)=x_{(n-1)i+l_1,
n(k-1)+l_2}(z)$ with $z=Y_{i,k}[l_1,l_2]$, and let $x_\beta(m)=
x_{n(k-1)+l_2,(n-1)i+l_1+1}(m)$. Then, $x_{\alpha+\beta}(l)=
x_{(n-1)i+l_1,(n-1)i+l_1+1}(l)$, and from the properties of $f_W$, 
we have $f_W(x_{\alpha+\beta}(l)g)=\psi(-l)f_W(g)$. Applying Corollary
\ref{cor1} several times in the indicated order, the above claim follows.

Hence, integral \eqref{local13} is equal to
\begin{equation}\label{local16}
\alpha(t)f_W(w_0t_0w_0^{-1})=\alpha(t)\delta_{B_r}^{\frac{n-1}{2}}
(t)\prod_{i=1}^rW_{\Theta_n}\begin{pmatrix} a_i&\\ &I_{n-1}
\end{pmatrix}
\end{equation}
Here $W_{\Theta_n}$ is the local Whittaker function associated to
the Theta function of the group $GL_n^{(n)}$. Also, we have 
$\delta_{B_r}^{\frac{n-1}{2}}(t)=\delta_{P_{n,r}}^{\frac{n-1}{2n}}(
w_0t_0w_0^{-1})$, where $P_{n,r}$ is the parabolic subgroup of $GL_{nr}$ whose Levi part is $GL_n\times GL_n\times \ldots\times GL_n$.
It follows from 
\cite{H} Propositions 5.1 and 5.3 that 
$W_{\Theta_n}\begin{pmatrix} a_i&\\ &I_{n-1}\end{pmatrix}=0$ unless
$|a_i|\le 1$ and $a_i=b_i^n$. In that case the value of the function is $|b_i|^{\frac{(n-1)^2}{2}}$. Notice that when $a_i=b_i^n$, then we have
$f_{\pi^{(n)}}(t)=\prod_{i=1}^r\chi_i^n(b_i)\delta_{B_r}^{n/2}
(\text{diag}(1,\ldots, 1,b_i,1,\ldots,1)$. Combing all this, integral
\eqref{local6} is equal to
\begin{equation}\label{local17}
\prod_{i=1}^r\int\limits_{|b_i|\le 1}\chi_i^n(b_i)|b_i|^{ns'+
\frac{n(n-2)(r-1)}{2}+\frac{(n-1)^2}{2}}db_i
\end{equation}
From this the Proposition follows.

\end{proof}

\section{ The Whittaker functional of the generating function}\label{bf}
In this section we compute the Whittaker functional of the function
$W_{nr}^{(n)}(g_0)$. Here the notations are as in Section \ref{gen},
but we assume that $r<n$. We make this assumption to get a precise proof of Conjecture 1.2 in \cite{B-F}. The case when $r\ge n$ is similar and will be dealt with in the next section. Embed $g\in GL_r$ in $GL_n$ as
$g\mapsto \text{diag}(g,I_{n-r})$. Let $g_0=\text{diag}(g,I_n,\ldots
,I_n)\in GL_{nr}$, where $I_n$ appears $r-1$ times.

Let $V_r$ denote the standard maximal unipotent subgroup of $GL_r$,
and let $\psi^{-1}_{V_r}$ denote the Whittaker character of $V_r$. See
equation \eqref{whit1} for the definition of $\psi_{V_r}$ . Let $W_{\Theta_n}^{(n)}$ denote the Whittaker function of the Theta function defined on 
$GL_n^{(n)}$. Our goal is to prove
\begin{theorem}\label{th1}
Assume that $r<n$. With the above notations, for all $g\in GL_r^{(n)}$, we have
\begin{equation}\label{main1}
\int\limits_{V_r}W_{nr}^{(n)}(v_0g_0)\psi^{-1}_{V_r}(v)dv=
W_{\Theta_n}^{(n)}\begin{pmatrix} g&\\ &I_{n-r}\end{pmatrix}
|\text{det}\ g|^{\frac{(n-1)(r-1)}{2}}
\end{equation}
\end{theorem}
\begin{proof}
We will consider the case when $r=n-1$. This is the hardest case. When
$r<n-1$ the computations are similar but simpler. Since we will use some of the notations introduced in the previous Sections, we will keep 
writing $r$ and $n$ even though we assume that $r=n-1$. By the Iwasawa decomposition, it is enough to prove identity \eqref{main1} for 
$g=t=\text{diag}(a_1,a_2,\ldots,a_{n-1})$. Notice, that from the left
invariant properties of $W_{nr}^{(n)}$ and $W_{\Theta_n}^{(n)}$, we may assume that $|a_i|\le 1$ for all $1\le i\le n-1$. 

We start by plugging integral \eqref{local3} into the left hand
side of identity \eqref{main1}. Doing so, we obtain the integral
\begin{equation}\label{whit11}
\int\limits_{V_r}\int\limits_{U_{nr}^0}
f(w_Jw_0uv_0t_0)\psi_{U_{nr}}(u)\psi^{-1}_{V_r}(v)dudv_0
\end{equation}
As in the proof of Proposition \ref{prop2}, we claim that integral
\eqref{whit11} is equal to integral
\begin{equation}\label{whit2}
\int\limits_{U_{nr}^0}
f(w_Jw_0u\delta_0t_0)\psi_{U_{nr}}(u)du
\end{equation}
where $\delta_0=\prod_{i=2}^{n-1}x_{i,i-1}(1)$. Here, the definition of  $U_{nr}^0$ and $x_{a,b}(m)$ are given before and after integral \eqref{local8}. To prove the above claim, we follow exactly the same steps as in the proof that integral \eqref{local7} is equal to 
integral \eqref{local8}. The only difference, is that because of
the character $\psi^{-1}_{V_r}$ in integral \eqref{whit11}, then for the
suitable variables in $V_r$, we need to use Lemma \ref{simple1}
with $\epsilon=-1$ and not with $\epsilon=0$ as in the proof of Proposition \ref{prop2}. This explains the element $\delta_0$. Next we
proceed as in the proof of Proposition \ref{prop2}. Following the
exact steps which showed that integral \eqref{local8} is equal to 
integral \eqref{local13}, we deduce that integral \eqref{whit2} is 
equal to
\begin{equation}\label{whit3}\notag
I_1=\alpha(t)\int\limits_{U_{nr}^3}f_W(w_0t_0w_0^{-1}u_3\delta_1(t))du_3
\end{equation}
Here the group $U_{nr}^3$ and $\alpha(t)$ are defined before integral \eqref{local13}, and we remind the reader that we assume that $r=n-1$.
Also, we have $\delta_1(t)=\prod_{i=2}^{n-1}x_{(i-1)n+1,(i-2)n+2}(a_i^{-1})$. This element is obtained by conjugating $\delta_0$ by 
$w_0$ and $t_0$.

At this point, for all $2\le j\le n-1$ we will introduce an integral
which we denote by $I_j$. To do that we first fix some notations. Let
$t_j=\text{diag}(A_{j,j},A_{j,j+1},\ldots, A_{j,n-1},I_n,\ldots,I_n)$
denote the torus element of $GL_{nr}$ where $A_{j,j}=\text{diag}
(a_1,\ldots,a_j,I_{n-j})$ and for all $j+1\le i\le n-1$ we define
$A_{j,i}=\text{diag}(I_j,a_i,I_{n-j-1})$. Notice that when $j=n-1$, we 
get $t_j=\text{diag}(A_{n-1,n-1},I_n,\ldots,I_n)=t_0$. Next we define
$\alpha_j(t)=|a_{j+1}a_{j+2}^2a_{j+3}^3\cdots a_{n-1}^{n-j-1}|^{n-2}$,
where we set $\alpha_{n-1}(t)=1$.
Finally, we define a set of subgroups $U_{n,j}$, and a set of characters $\psi_{U_{n,j}}$ defined on these groups. The definition is 
inductive, so we we start with $U_{n,2}$. Consider the 
group $U_{nr}^3$ with $r=n-1$, as was defined right after equation 
\eqref{mat2}. Let $U_{n,2}$ denote the subgroup of $U_{nr}^3$ with the extra condition that $Y_{n-1,i}=0$ for all $1\le i\le n-2$. Assuming we defined $U_{n,j-1}$ we define $U_{n,j}$ as
the subgroup of $U_{n,j-1}$ consisting of matrices of the form \eqref{mat2} such that $Y_{n-j+1,i}=0$ for all $1\le i\le n-j$, and
also satisfies the condition $Y_{i,l}[b,j]=0$ for all $2\le i\le n-j$,
$1\le l\le i-1$ and $1\le b\le n$. The character $\psi_{U_{n,j}}$ is
defined as follows. For $u\in U_{n,j}$ written as in equation \eqref{mat2}, we set $\psi_{U_{n,j}}(u)=\psi(\sum_{i=2}^{n-j}
Y_{i,i-1}[j-1,j+1])$.

With these notations, for all $2\le j\le n-1$ we set
\begin{equation}\label{whit4}\notag
I_j=\alpha_j(t)\int\limits_{U_{n,j}}f_W(t_ju)\psi_{U_{n,j}}(u)du
\end{equation}
We will prove that $I_2=I_1$, and that for all $2\le j\le n-1$, we have $I_j=I_{j-1}$. This will complete the proof of the Theorem. Indeed, proving the above implies that the left hand side
of equation \eqref{main1} is equal to $I_{n-1}$. Since $\alpha_{n-1}(t)=1$, the group $U_{n,n-1}$ is the trivial group, and $t_{n-1}=t_0$,
we deduce that $I_{n-1}=f_W(t_0)$. But as in equation \eqref{local16} we obtain that $f_W(t_0)$ equals the right hand side of equation 
\eqref{main1}. 

We prove that $I_2=I_1$. Since $|a_i|\le 1$, we obtain the following
Iwasawa decomposition $\delta_1(t)=\prod_{i=2}^{n-1}x_{(i-2)n+2,(i-1)n+1}(a_i)\prod_{i=2}^{n-1}h_i(a_i)k$. Here $k\in K_{nr}$, and we have $\prod_{i=2}^{n-1}h_i(a_i)=\text{diag}(B_{2,1},B_{2,2},\ldots,B_{2,n-1})$. Here
$B_{2,1}=\text{diag}(1,a_2,I_{n-2})$, for $2\le i\le n-2$ we have
$B_{2,i}=\text{diag}(a_i^{-1},a_{i+1},I_{n-2})$ and
$B_{2,n-1}=\text{diag}(a_{n-1}^{-1},I_{n-1})$. Conjugating in $I_1$
the matrix $\delta_1(t)k^{-1}$ to left, and using the left invariant
properties of $f_W$, we obtain by matrix multiplication
\begin{equation}\label{whit5}
I_1=\alpha_2(t)\int\limits_{U_{nr}^3}f_W(t_2u_3)\psi_{U_{nr}^3}(u_3)du_3
\end{equation}
Here we use the fact that $w_0t_0w_0^{-1}\prod_{i=2}^{n-1}h_i(a_i)=t_2$. The factor of $|a_2a_3\ldots a_{n-1}|
^{-(n-2)}$ is obtained from a change of variables when we conjugate the torus $\prod_{i=2}^{n-1}h_i(a_i)$ across $U_{nr}^3$. The product of this factor by $\alpha_1(t)$ is equal to $\alpha_2(t)$. The character $\psi_{U_{nr}^3}$ is defined as 
follows. For $u_3\in U_{nr}^3$ define $\psi_{U_{nr}^3}(u_3)=\psi(\sum_{i=2}^{n-1}Y_{i,i-1}[1,3])$. To complete the proof that $I_2=I_1$,
we need to show that we can restrict the support of integration from 
$U_{nr}^3$ to $U_{n,2}$. In other words, we need to show that for all $1\le i\le n-2$, the integration over all variables in $Y_{n-1,i}$ is
in $K_{nr}$. This is done as in the proof of Proposition \ref{prop2}
while showing that integral \eqref{local13} reduces to the left
hand side of identity \eqref{local16}. Indeed, from the definition of the torus $t_2$, given a variable $Y_{n-1,i}[l_1,l_2]$, we can find
a one dimensional unipotent subgroup $x_\beta(m)$ so that we can apply
Corollary \ref{cor1}. Thus $I_2=I_1$.

The next step is to prove that $I_j=I_{j-1}$. The first step is to 
prove that we can integrate over a smaller unipotent group. Let $U_{n,j-1,1}$ denote the subgroup of $U_{n,j-1}$ consisting of
all matrices which also satisfies $Y_{i,l}[b,j]=0$ for all 
$3\le i\le n-j+1$, $1\le l\le i-2$ and $1\le b\le j-1$. To show that we can reduce the domain of integration from $U_{n,j-1}$ to 
$U_{n,j-1,1}$ we apply Corollary \ref{cor1}. In the notations of this
Corollary, let $x_\alpha(z)=x_{n(i-1)+b, n(l-1)+j}(z)$ with 
$z=Y_{i,l}[b,j]$, and let $x_\beta(m)=x_{nl+j-2,n(i-1)+b}(m)$. Notice
that in this case the root $\alpha+\beta$ corresponds to the one
dimensional unipotent subgroup $x_{nl+j-2,n(l-1)+j}(c)$, which is a
subgroup of $U_{n,j-1,1}$. Moreover, the character  $\psi_{U_{n,j-1}}$
is not trivial on this subgroup. Hence, the conditions of the Corollary \ref{cor1} are satisfied. 
We mention that the order for which we apply this Corollary is important. We first vary $3\le i\le n-j+1$ and fix $l=1$. Then we 
repeat the same process with $l=2$ and so on. 

The second step is to show that $I_{j-1}$ is equal to
\begin{equation}\label{whit6}
\alpha_{j-1}(t)\int\limits_{U_{n,j-1,2}}f_W(t_{j-1}u\delta_{j-1}(t))
\psi_{U_{n,j-1}}(u)du
\end{equation}
Here $U_{n,j-1,2}$ is the subgroup of $U_{n,j-1,1}$ which satisfies 
the condition that $Y_{i,i-1}[j-1,j]$ for all $2\le i\le n-j+1$. The
matrix $\delta_{j-1}(t)=\prod_{i=2}^{n-j+1}x_{(i-1)n+j-1,n(i-2)+j}(a_{j+i-2}^{-1})$. To derive integral \eqref{whit6} we apply Lemma 
\ref{simple1} with $x_\alpha(z)=x_{n(i-1)+j-1,n(i-2)+j}(z)$ with
$z=Y_{i,i-1}[j-1,j]$ and $x_\beta(l)=x_{(i-1)n+j-2,(i-1)n+j-1}(l)$.
The next step is to perform an Iwasawa decomposition for $\delta_{j-1}(t)$ in integral \eqref{whit6}. This is done as with $\delta_1(t)$ and we obtain 
$$\delta_{j-1}(t)=\prod_{i=2}^{n-j+1}x_{n(i-2)+j,(i-1)n+j-1}(a_{j+i-2})\prod_{i=2}^{n-j+1}h_i'(a_{j+i-2})k$$ 
where $k\in K_{nr}$. Here $\prod_{i=2}^{n-j+1}h_i'(a_{j+i-2})=\text{diag}(B_{j-1,1},B_{j-1,2},\ldots,B_{j-1,n-j+1},I_n,\ldots,I_n)$,
where $B_{j-1,1}=\text{diag}(I_{j-1},a_j,I_{n-j})$, $B_{j-1,i}=\text{diag}(I_{j-2},a_{j+i-2}^{-1},a_{j+i-1},I_{n-j})$ for $2\le i\le n-j$, and $B_{j-1,n-j+1}=\text{diag}(I_{j-2},a_{n-1}^{-1},
I_{n-j+1})$. Plugging  this into integral \eqref{whit6} and conjugate
the matrix $\delta_{j-1}(t)k^{-1}$ to the left, we obtain
\begin{equation}\label{whit7}
\alpha_{j}(t)\int\limits_{U_{n,j-1,2}}f_W(t_{j}u)
\psi_{U_{n,j}}(u)du
\end{equation}
Here, we obtain the factor of $|a_ja_{j+1}\ldots a_{n-1}|^{-(n-2)}$ from the conjugation of the toral part of $\delta_{j-1}(t)k^{-1}$
across $U_{n,j-1,2}$. This combined with 
$\alpha_{j-1}(t)$ gives the factor $\alpha_{j}(t)$ in integral
\eqref{whit7}. Notice also that the conjugation by the unipotent part
of $\delta_{j-1}(t)k^{-1}$ changes the additive character to 
$\psi_{U_{n,j}}$. This is well defined. Indeed, notice that $U_{n,j}$ is a subgroup $U_{n,j-1,2}$
and we can view $\psi_{U_{n,j}}$ as a character of $U_{n,j-1,2}$ by extending it trivially.  Finally, we have the
identity $t_{j-1}\prod_{i=2}^{n-j+1}h_i'(a_{j+i-2})=t_j$. To show that
integral \eqref{whit7} equals $I_j$, we need to show that we may restrict the domain of integration from $U_{n,j-1,2}$ to $U_{n,j}$. 
We do so using Corollary \ref{cor1}. Indeed, the group $U_{n,j}$ is 
the subgroup of $U_{n,j-1,2}$ obtained by setting $Y_{n-j+1,l}=0$ for
all $1\le l\le n-j$ and $Y_{i,i-1}[b,j]=0$ for all $2\le i\le n-j$
and $1\le b\le j-2$. To show that we may restrict the integration 
over $U_{n,j-1,2}$ to the subgroup obtained by setting $Y_{n-j+1,l}=0$, we argue in a similar way as in the reduction from the group 
$U_{nr}^3$ to $U_{n,2}$ as was done right after integral \eqref{whit5}. Then, finally to obtain the group $U_{n,j}$ we use 
Corollary \ref{cor1} with $x_\alpha(z)=x_{n(i-1)+b,n(i-2)+j}(z)$
where $z=Y_{i,i-1}[b,j]$ and $x_{n(i-2)+j,n(i-1)+b+1}(m)$. Here 
$2\le i\le n-j$ and $1\le b\le j-2$.

\end{proof}

\section{ The case when $r\ge n$}
As mentioned in the introduction, the authors of \cite{B-F} were well aware 
that the situation when $r\ge n$ is similar. Since they do not specify
this case explicitly, we briefly mention the global constructions and
show how a similar result to Theorem \ref{th1} holds in this case.

Assume first that $r>n$.
Let $\pi^{(n)}$ denote a cuspidal representation of the group 
$GL_r^{(n)}({\bf A})$. Let $\Theta_{n}^{(n)}$ denote the Theta representation of the group $GL_n^{(n)}({\bf A})$. Then we consider the global integral \eqref{global100} introduced in the introduction.
The group $V_{r,n}$ is defined as follows. Recall that $V_r$ is the standard maximal unipotent subgroup of $GL_r$. Then, $V_{r,n}$ is the
subgroup of $V_r$ consisting of all matrices $v=(v_{i,j})\in V_r$ such
that $v_{i,j}=0$ for all $2\le j\le n+1$. The character $\psi_{V_{r,n}}$ is defined by $\psi_{V_{r,n}}(v)=\psi(v_{n+1,n+2}+
v_{n+2,n+3}+\cdots +v_{r-1,r})$. It follows from the cuspidality of
$\phi$ that integral \eqref{global100} converges for all $s$.
A similar unfolding as in \cite{B-F} Section 2, implies that for 
$\text{Re}(s)$ large, integral \eqref{global100} is equal to
integral \eqref{global110}.

Next we consider the case when $r=n$. In this case the global 
integral is given by
\begin{equation}\label{global12}
\int\limits_{Z_r({\bf A})GL_r(F)\backslash GL_r({\bf A})}\phi(g)\overline{\theta(g)}E(g,s)dg
\end{equation}
Here $Z_r$ is the subgroup of $Z$, the center of $GL_r$, which consists of scalar matrices which are $r$ powers. For simplicity 
we assume that all representations have a trivial central character.
Also, $E(g,s)$ is the Eisenstein series defined on the group 
$GL_r({\bf A})$ and is associated with the induced representation
$Ind_{P({\bf A})}^{GL_r({\bf A})}\delta_P^s$. Here $P$ is the maximal
parabolic subgroup of $GL_r$ whose Levi part is $GL_{r-1}\times GL_1$.
Unfolding this integral, by first unfolding the Eisenstein series, we
obtain for $\text{Re}(s)$ large, that integral \eqref{global12} is 
equal to
\begin{equation}\label{global13}
\int\limits_{Z_r({\bf A})V_r({\bf A})\backslash GL_r({\bf A})}W_\phi(g)\overline{W_\theta(g)}f(g,s)dg
\end{equation}
Here $f(g,s)$ is a section in the above induced representation.

Next we study the local unramified computation corresponding to the
integrals \eqref{global110} and \eqref{global13}. Since the Whittaker
coefficient of the representation $\pi^{(n)}$ is not factorizable, 
it is not clear that these integrals are Eulerian. However, as explained in \cite{B-F}, if we can prove similar results to 
Proposition \ref{prop2} and to Theorem \ref{th1}, the so call New
way method would imply that these integrals are indeed factorizable. 
As for Proposition \ref{prop2}, it is clear that it holds for all
values of $r$ and $n$. 

As for Theorem \ref{th1}, this is not the case for all matrices $g\in GL_r^{(n)}$.  Assuming $r\ge n$, denote by $T_0$ the subgroup of $GL_r$ which consists of all diagonal matrices $t=\text{diag}(a_1,a_2,\ldots,a_r)$ such  that $|a_i|\le 1$ for $1\le i\le n$ and $|a_i|=1$
for all $n+1\le i\le r$. Let $T_0^{(n)}$ denote the inverse image of
$T_0$ inside $GL_r^{(n)}$. Let $GL_{r,0}^{(n)}$ denote all elements
$g\in GL_r^{(n)}$ which can be written as $g=vtk$ where $v\in V_r$, 
$t\in T_0^{(n)}$ and $k\in K_r$. Here $K_r$ is the standard maximal compact subgroup of $GL_r$. With these notations we have 
\begin{theorem}\label{th2}
Assume that $r\ge n$. Then, for all $g\in GL_{r,0}^{(n)}$, we have
\begin{equation}\label{main2}
\int\limits_{V_r}W_{nr}^{(n)}(v_0g_0)\psi^{-1}_{V_r}(v)dv=
W_{\Theta_n}^{(n)}(g)|\text{det}\ g|^{\frac{(n-1)(r-1)}{2}}
\end{equation}
\end{theorem}
\begin{proof}
The proof of this Theorem is the same as the proof of Theorem \ref{th1},  and so we will only indicate the end result. Using the Iwasawa decomposition, we may assume that $g_0=t_0=\text{diag}(t,I_r,
\ldots,I_r)$ where $t=\text(a_1,\ldots,a_r)$.

Defining
similar integrals $I_j$ as in Theorem \ref{th1}, we prove that the
left hand side of integral \eqref{main2} is equal to $f_W(t_0')$
where $t_0'=\text{diag}(A_1,A_2,\dots,A_{r-n+1},I_n,\dots,I_n)$. Here
$A_1=\text{diag}(a_1,\ldots,a_n)$ and for all $2\le i\le r-n+1$ we have $A_i=\text{diag}(I_{n-1},a_{n+i-1})$. Applying the factorization
of equation \eqref{local16}, we obtain the identity 
$$f_W(t_0')=\delta_{P_{r,n}}^{\frac{n-1}{2n}}(t_0')\prod_{i=1}^{r-n+1}
W_{\Theta_n}^{(n)}(A_i)$$
From the properties of the Whittaker function, we deduce that for all
$2\le i\le r-n+1$ we have $W_{\Theta_n}^{(n)}(A_i)=0$ unless 
$|a_{n+i-1}|=1$. From this the Theorem follows.

\end{proof}

Notice that this Theorem is enough
to prove that the corresponding local integrals of integrals 
\eqref{global110} and \eqref{global13}, are Eulerian. Indeed, the 
local version of integral \eqref{global110} is given by
\begin{equation}\label{global14}
\int\limits_{V_r\backslash GL_r}W_\phi\begin{pmatrix} g&\\ &I_{r-n}\end{pmatrix}\overline{W_\theta(g)}|\text{det} g|^{s-\frac{r-n}{2}}dg
\end{equation}
Here $\phi$ is a vector in the local component of $\pi^{(n)}_\nu$
where $\nu$ is a place where all data is unramified. Similarly for
$\theta$. Also, $W_\phi$ is any local Whittaker functional defined
on the representation $\pi^{(n)}_\nu$. Similarly $W_\theta$ is
the Whittaker functional defined on the space of 
$\Theta_{n,\nu}^{(n)}$. It is known that for the representation 
$\Theta_{n,\nu}^{(n)}$ this Whittaker functional is unique ( see \cite{K-P}). However, this need not be the case for the representation
$\pi^{(n)}_\nu$. 

Applying the Iwasawa decomposition to the quotient 
$V_r\backslash GL_r$, the domain of integration in integral \eqref{global14} is reduced to the torus $T_r$ of $GL_r$. However,
because of the Whittaker functional properties, $W_\phi\begin{pmatrix}
t&\\&I_{r-n}\end{pmatrix}$ is zero unless $t\in T_0$.
Hence, we can apply
Theorem \ref{th2} to deduce that integral \eqref{global110} is indeed
Eulerian.

A similar argument applies to integral \eqref{global13}. Indeed, 
using the properties of the Whittaker function, we
can choose representatives for the quotient $Z_r\backslash T_r$ to
be in the group $T_0$. Hence, once again we can apply Theorem \ref{th2}.

\end{document}